\documentclass[12pt,a4paper,oneside]{amsart}
\usepackage{tikz-cd}
\usepackage{cite}
\usepackage[english]{babel}
\usepackage{geometry}
\usepackage{amsmath}
\usepackage{amssymb}

\newtheorem{thm}{Theorem}[section]
\newtheorem{lem}[thm]{Lemma}
\newtheorem{cor}[thm]{Corollary}
\newtheorem{exa}[thm]{Example}
\newtheorem{prop}[thm]{Proposition}
\newtheorem{Def}[thm]{Definition}
\newtheorem{rem}[thm]{Remark}

\long\def\comment#1{\relax}

\let\takeaway=\setminus
\newcommand{\height}{\mathop{\mathrm{ht}}}  

\newcommand{\pset}{\mathop{\mathcal{P}}}  
\newcommand{\Int}{\mathop{\mathrm{Int}}} 
\newcommand{\R}{{\mathcal{R}}} 
\newcommand{\fring}[2]{\R(#1,#2)} 

\def\natn{\mathbb N} 
\def\intz{\mathbb Z} 

\def\I{\mathcal I} 
\def\J{\mathcal J} 
\def\Q{\mathcal Q} 
\def\F{\mathcal F} 
\def\G{\mathcal G} 
\def\U{\mathcal U} 
\def\Z{\mathcal Z} 
\def\MU{M_{\mathcal U}}
\def\PU{P_{\mathcal U}}
\def\IF{I_{\mathcal F}}
\def\JF{J_{\mathcal F}}
\def\IG{I_{\mathcal G}}

\def\MF{M_{\mathcal F}}

\long\def\comment#1{\relax}
\long\def\commented_out#1{\relax}

\def\ufprimes_prime_name{Lemma}

\author[S.~Frisch]{Sophie Frisch}
\address{\hspace{-12pt}Department of Analysis and Number Theory (5010) \\
Technische Universit\"at Graz \\
Kopernikusgasse 24 \\
8010 Graz, Austria}
\email{frisch@math.tugraz.at}
\title[Spectrum of rings of functions]%
{On the spectrum of rings of functions}
\thanks{This research was supported by the Austrian Science Fund FWF
grant P27816-N26}

\begin{document}
\hbox to \textwidth{\ \hfill \ }
\vskip-2cm
To appear in J.~Pure Appl. Algebra 
\vskip1cm

\addtolength{\itemsep}{2pt}
\addtolength{\parsep}{2pt}
\addtolength{\itemsep}{2pt}
\begin{abstract}
For $D$ a domain and $E\subseteq D$, we investigate the prime spectrum
of rings of functions from $E$ to $D$, that is, of rings contained in 
$\prod_{e\in E} D$ and containing $D$. Among other things, we characterize,
when $M$ is a maximal ideal of finite index in $D$, those prime ideals lying
above $M$ which contain the kernel of the canonical map to 
$\prod_{e\in E} (D/M)$ as being precisely the prime ideals corresponding 
to ultrafilters on $E$. 
We give a sufficient condition for when all primes above $M$ are of this
form and thus establish a correspondence to the prime spectra
of ultraproducts of residue class rings of $D$. As a corollary, we
obtain a description using ultrafilters, differing from Chabert's 
original one which uses elements of the $M$-adic completion, of the prime
ideals in the ring of integer-valued polynomials $\Int(D)$ lying above 
a maximal ideal of finite index. 
%
\end{abstract}

\hfuzz=1pt
\maketitle

\section{Introduction}
Let $D$ be an integral domain, $E\subseteq D$, and $\R$ a subring of 
$\prod_{e\in E} D$, containing $D$. The elements of $\R$ 
can be interpreted as functions from $E$ to $D$ and, consequently,
we call $\R$ a ring of functions from $E$ to $D$.
We will investigate the prime spectra of such rings of functions. 
We obtain, for quite general $\R$, a partial description of 
the prime spectrum, cf.\ Theorems~\ref{maxideal} and \ref{finitethm}, 
and in special cases a complete 
characterization, cf.\ Corollary~\ref{finalcor}. 

Our motivation is the spectrum of a ring of integer-valued polynomials:
For $D$ an integral domain with quotient field $K$, let
$\Int(D) = \{ f\in K[x]\mid f(D)\subseteq D\}$
be the ring of integer-valued polynomials on $D$. 
More generally, when $K$ is understood, we let
$\Int(A,B) = \{ f\in K[x]\mid f(A)\subseteq B\}$
for $A,B\subseteq K$.

If $D$ is a Noetherian one-dimensional domain, a celebrated theorem 
of Chabert \cite[Ch.~V]{CahCha97IVP} states that
every prime ideal of $\Int(D)$ lying over a maximal ideal $M$ of
finite index in $D$ is maximal and of the form
\[ M_\alpha = \{f\in \Int(D) \mid f(\alpha)\in \hat M\}, \]
where $\alpha$ is an element of the $M$-adic completion $\hat D_M$
of $D$ and $\hat M$ the maximal ideal of $\hat D_M$.

In fact, Chabert showed two separate statements independently -- 
both under the assumption that $D$ is Noetherian and one-dimensional 
and $M$ a maximal ideal of finite index of $D$:
\begin{enumerate}
\item\label{characterization}
Every maximal ideal of $\Int(D)$ containing $\Int(D,M)$ is
of the form $M_\alpha$ for some $\alpha\in \hat D_M$.
\item\label{containing}
Every maximal ideal of $\Int(D)$ lying over $M$ contains $\Int(D,M)$.
\end{enumerate}

For a simplified proof of Chabert's result, see \cite{Fri13IVA},
Lemma~4.4 and the remark following it.

We will show that a modified version of statement (\ref{characterization}) 
holds in far greater generality, for rings of functions. The modification 
consists in replacing elements of the $M$-adic completion by ultrafilters. 

Whether (\ref{containing}) holds or not for a particular $D$ and a
particular subring of $D^E$ will have to be examined separately. It is, 
in some sense, a question of density of the subring in the product 
$\prod_{e\in E}D$.

\comment{
for an ultrafilter $\U$ on $D$ we define
\[ M_{\U} = \{f\in \Int(D) \mid f^{-1}(M)\in \U\}, \]
where $f^{-1}(M)=\{d\in D\mid f(d)\in M\}$. It turns out that,
for every maximal ideal $M$ of finite index in an arbitrary domain $D$, 
every maximal ideal $\mathcal{Q}$ of $\Int(D)$ containing $\Int(D,M)$
is of the form $M_{\U}$ for an ultrafilter $\U$ on $D$.
}

We will work in the following setting:

\begin{Def}\label{funcring-def}
Let $D$ be a commutative ring and $E\subseteq D$. Let $\R$ be a 
commutative ring and $\varphi\colon \R\rightarrow \prod_{e\in E}D$
a monomorphism of rings. $\varphi$ allows us to interpret the elements
of $\R$ as functions from $E$ to $D$. 

If all constant functions are contained in $\varphi(\R)$, we call the 
pair $(\R,\varphi)$ a ring of functions from $E$ to $D$. We use
$\R=\fring{E}{D}$ (where $\varphi$ is understood) to denote 
a ring of functions from $E$ to $D$.
\end{Def}

\begin{rem}\label{constants}
For our considerations it is vital that $\R=\fring{E}{D}$
contain all constant functions, because we will make extensive
use of the following fact: when $\I$ is an ideal of $\R=\fring{E}{D}$,
$f\in \I$ and $g\in D[x]$ a polynomial with zero constant term, then
$g(f)\in \I$, and similarly, if $g$ is a polynomial in several variables
over $D$ with zero constant term, and an element of $\I$ is substituted
for each variable in $g$, then, an element of $\I$ results.
\end{rem}

Let us note that considerable research has been done on the
spectrum of a power of a ring $D^E=\prod_{d\in E}D$ or a product
of rings $\prod_{e\in E}D_e$. Gilmer and Heinzer 
\cite[Prop.~2.3]{GilHei94Imb} have determined the 
spectrum of an infinite product of local rings, and
Levy, Loustaunau and Shapiro \cite{LeLoSh91PSIPZ} that
of an infinite power of $\intz$. Our focus here is not on the
full product of rings, but on comparatively small subrings and 
the question of how much information about the spectrum of a ring 
can be obtained from its embedding in a power of a domain.

One ring can be embedded in different products: $\Int(D)$ can be
seen as a ring of functions from $K$ to $K$ as well as a ring of
functions from $D$ to $D$. We will glean a lot more information
about the spectrum of $\Int(D)$ from the second interpretation
than from the first. 

\section{Prime ideals corresponding to ultrafilters}

Let $\R=\fring{E}{D}$ be a ring of functions from $E$ to $D$ as in
Definition~\ref{funcring-def}. 
We will now make precise the concept of ideals corresponding to
ultrafilters, and the connection to ultraproducts
$\prod^{\U}_{e\in E} (D/M)$, where $M$ is a maximal ideal of $D$,
and $\U$ an ultrafilter on $E$.
First a quick review of filters, ultrafilters and ultraproducts:

\begin{Def}\label{filter-def}
Let $S$ be a set. A non-empty collection $\F$ of subsets of $S$
is called a filter on $S$ if
\begin{enumerate}
\item
$\emptyset \notin \F$.
\item
$A, B\in \F$ implies $A\cap B\in \F$.
\item
$A\subseteq C\subseteq S$ with $A\in \F$ implies $C\in \F$.
\end{enumerate}
A filter $\F$ on $S$ is called an ultrafilter on $S$ if, for
every $C\subseteq S$, either $C\in \F$ or $S\setminus C\in \F$.
\end{Def}

Let $S$ be a fixed set and $\pset(S)$ its power-set. 
For $C\in\pset(S)$, a {\em superset} of $C$ is a set $D\in \pset(S)$ with 
$C\subseteq D\subseteq S$.
A collection $\mathcal{C}$ of subsets of $S$ is 
said to have the {\em finite intersection property} if the intersection
of any finitely many members of $\mathcal{C}$ is non-empty. 

\begin{rem}\label{finite-intersection-property}
Clearly,
a necessary and sufficient condition for $\mathcal{C}\subseteq \pset(S)$ to 
be contained in a filter on $S$ is that $\mathcal{C}$ satisfies the finite 
intersection property. 
If the finite intersection property is satisfied, then
the supersets of finite intersections of members of $\mathcal{C}$ 
form a filter. 
\end{rem}

Although, strictly speaking, we do not need ultraproducts to prove 
our results, we will nevertheless introduce them, because they provide 
context, in particular to Lemma~\ref{ufprimes-prime}, and to sections~3
and~5. 

\begin{Def}\label{ultraproduct-def}
Let $S$ be an index set and $\U$ an ultrafilter on $S$. Suppose
we are given, for each $s\in S$, a ring $R_s$. Then the
ultraproduct of rings $\prod_{s\in S}^{\U} R_s$ is defined as
the direct product $\prod_{s\in S}R_s$ modulo the congruence
relation 
\[(r_s)_{s\in S} \sim (t_s)_{s\in S} \quad\Longleftrightarrow\quad
\{s\in S\mid r_s=t_s\}\in \U.\]
\end{Def}

Ultraproducts of other algebraic structures are defined analogously.
The usefulness of ultraproducts is captured by the Theorem of \L o\'s
(cf.~\cite[Chpt.~3.2]{GoJu95IncPh} or \cite[Prop~1.6.14]{Hin05mathlog})
which states that an ultraproduct $\prod_{s\in S}^{\U} R_s$ satisfies
a first-order formula if and only if the set of indices $s$ for which
$R_s$ satisfies the formula is in $\U$. Here first-order formula means
a formula in the first-order language whose only non-logical symbols
(apart from the equality sign) are symbols for the algebraic operations; 
for instance, $+$ and $\cdot$ in the case of an ultraproduct of rings.

\begin{Def}\label{ufprime-def}
Let $D$ be a domain, $E\subseteq D$, $\R=\fring{E}{D}$ a ring of
functions, $I$ an ideal of $D$ and $\F$ a filter on $E$.

For $f\in \fring{E}{D}$, we 
let $f^{-1}(I)=\{e\in E\mid f(e)\in I\}$ and define
\[ \IF = \{f\in \fring{E}{D} \mid f^{-1}(I)\in\F\} \]
\end{Def}

\begin{rem}\label{obviousinclusions} 
Let everything as in Definition~\ref{ufprime-def},
$I,J$ ideals of $D$ and $\F,\G$ filters on $E$. 
Some easy consequences of Definition~\ref{ufprime-def} are:
\begin{enumerate}
\item
If $I\neq D$ then $\IF\neq \R$.
\item
$\IF$ is an ideal of $\R$ containing
$\fring{E}{I}=\{f\in\R\mid f(E)\subseteq I\}$.
\item
$I\subseteq J\Longrightarrow \IF\subseteq \JF$
\item
$\F\subseteq \G\Longrightarrow \IF\subseteq \IG$
\end{enumerate}
\end{rem}

\begin{lem}\label{ufprimes-prime} 
Let $D$ be a domain, $E\subseteq D$, and $\R=\fring{E}{D}$ a ring of
functions from $E$ to $D$.  

Then for every prime ideal $P$ of $D$ and every ultrafilter $\U$ on $E$,
$\PU$ is a prime ideal of $\mathcal{R}$.
\end{lem}

\begin{proof} Easy direct verification: let $fg\in \PU$; 
because $P$ is a prime ideal of $D$, the inverse image of $P$ under 
$f\cdot g$ is the union of $f^{-1}(P)$ and $g^{-1}(P)$. If the union
of two sets is in an ultrafilter, then one of them must be in the
ultrafilter. Therefore, $f\in \PU$ or $g\in \PU$. Also, $\PU$ cannot
be all of $\mathcal{R}$ because it doesn't contain the constant 
function $1$.
\end{proof}

One way of looking at $\PU$ is by
considering the following commuting diagram of ring-homomorphisms,
where $\pi$ and $\pi_1$ mean applying the canonical projection
in each factor of the product, and $\sigma$ and $\sigma_1$ mean
factoring through the defining congruence relation of an ultraproduct.


\[
\begin{tikzcd}
\R \arrow[r, "\varphi"]&
\prod_{e\in E}D \arrow[r, "\sigma_1"]\arrow[d, "\pi"] &
\prod_{e\in E}^{\U}D\arrow[d, "\pi_1"]\\
&\prod_{e\in E}(D/P) \arrow[r, "\sigma"]&
\prod_{e\in E}^{\U}(D/P)
\end{tikzcd}
\]

\comment{
\[
\begin{tikzcd}
\prod_{e\in E}D \arrow[r, "\sigma_1"]\arrow[d, "\pi_2"] &
\prod_{e\in E}^{\U}D\arrow[d, "\pi_1"]\\
\prod_{e\in E}(D/P) \arrow[r, "\sigma_2"]&
\prod_{e\in E}^{\U}(D/P)
\end{tikzcd}
\]
}

$\PU$ is the kernel of the following composition of ring homomorphisms:
\[\varphi\colon \mathcal{R}\rightarrow \prod_{e\in E}D\]
followed by the canonical projection
\[\pi \colon \prod_{e\in E}D\rightarrow \prod_{e\in E}(D/P)\]
and the canonical projection 
\[\sigma \colon \prod_{e\in E}(D/P)\rightarrow \prod^{\U}_{e\in E}(D/P)\]

Since $D/P$ is an integral domain, any ultraproduct of copies of
$D/P$ is also an integral domain, by the Theorem of \L o\'s. 
Therefore $(0)$ is a prime ideal of $\prod^{\U}_{e\in E}(D/P)$ and
hence $\PU$ a prime ideal of $\mathcal{R}$.  
We also see that $\PU$ is the inverse image of a prime ideal of
$\prod_{e\in E} D$ under $\varphi$, and further,
of a prime ideal of the ultraproduct $\prod^{\U}_{e\in E} D$
under $\sigma_1\circ\varphi$.

\section{The set of zero-loci mod $M$ of an ideal of the ring of functions}

As before, $D$ is a domain with quotient field
$K$, $E\subseteq D$ and $\R=\fring{E}{D}$ a ring of 
functions from $E$ to $D$ as in Def.~\ref{funcring-def}. Especially,
recall from Def.~\ref{funcring-def} that $\R$ is assumed to contain 
all constant functions.

\begin{Def}\label{zero-locus-def}
For $M\subseteq D$ and $f\in\R=\fring{E}{D}$, let
\[ f^{-1}(M)=\{e\in E\mid f(e)\in M\}. \]

For an  ideal $M$ of $D$ and an ideal $\I$ of $\R$, let
\[  \Z_M(\I) = \{f^{-1}(M)\mid f\in\I\} \]
\end{Def}

Recall from Def.~\ref{ufprime-def} that for a filter $\F$ on $E$,
\[ \MF = \{f\in \fring{E}{D} \mid f^{-1}(M)\in\F\} \]

\begin{rem}\label{ZMproperties} Note that the above definition implies
\begin{enumerate}
\item
$\I\subseteq \J \Longrightarrow \Z_M(\I) \subseteq \Z_M(\J)$
\item
$\I\subseteq \MF \Longleftrightarrow \Z_M(\I)\subseteq \F$
\end{enumerate}
\end{rem}

\begin{lem}
Let $M$ be an ideal of $D$ and $\I$ an ideal of $\R$.
The following are equivalent:
\begin{description}
\item[(a)]
There exists a filter $\F$ on $E$ such that $\I\subseteq M_{\F}$.
\item[(b)]
$\Z_M(\I)$ satisfies the finite intersection property. 
\end{description}
\end{lem}

\begin{proof}
If $\I\subseteq M_{\F}$, then $\Z_M(\I)$ is contained in ${\F}$
and hence satisfies the finite intersection property. Conversely,
if $\Z_M(\I)$ satisfies the finite intersection property then,
by Remark~\ref{finite-intersection-property},
the supersets of finite intersections of sets in $\Z_M(\I)$ form a 
filter ${\F}$ on $E$ for which $\Z_M(\I)\subseteq \F$ and hence 
$\I\subseteq M_{\F}$.
\end{proof}

In the case where $\fring{E}{D}=\prod_{e\in E}D$ is the ring
of all functions from $E$ to $D$, much more can be said; see 
the papers by Gilmer and Heinzer \cite[Prop.~2.3]{GilHei94Imb} 
(concerning local rings) and Levy, Loustaunau and Shapiro 
\cite{LeLoSh91PSIPZ} (concerning $D=\intz$).

For a field $K$ that is not algebraically closed, 
we will need, for an arbitrary $n\ge 2$, an $n$-ary form that has 
no zero but the trivial one. For this purpose, recall
how to define a norm form: if $L:K$ is an $n$-dimensional
field extension, multiplication by any $w\in L$ is a 
$K$-endomorphism $\psi_w$ of $L$. For a fixed choice of
a $K$-basis of $L$, map every $w\in L$ to the determinant
of the matrix of $\psi_w$ with respect to the chosen basis.
This mapping, regarded as a function of the coordinates of $w$ 
with respect to the chosen basis, is easily seen to be an 
$n$-ary form that has no zero but the trivial one.

\begin{lem}\label{intersections-notalgcl}
Let $M$ be a maximal ideal of $D$ such that $D/M$ 
is not algebraically closed. 
Then for every ideal $\mathcal{I}$ of $\R=\fring{E}{D}$,
$\Z_M(\I)$ is closed under finite intersections.
\end{lem}

\begin{proof}
Given $f,g\in\I$, we show that there exists $h\in\I$
with 
\[ h^{-1}(M) = f^{-1}(M)\cap g^{-1}(M) . \] 
Consider any finite-dimensional non-trivial field 
extension of $D/M$, and let $n$ be the degree of the extension.
The norm form of this field extension is a homogeneous polynomial
in $n\ge 2$ indeterminates whose only zero in $(D/M)^n$ is the trivial one.
By identifying $n-1$ variables, we get a binary form
$\bar s\in (D/M)[x,y]$ with no zero in $(D/M)^2$ other than $(0,0)$.
Let $s\in D[x,y]$ be a binary form that reduces to $\bar s$
when the coefficients are taken mod $M$.

Now, given $f$ and $g$ in $\I$, we set $h=s(f,g)$. 
By the fact that $\R$ contains all constant functions, $h$ is in $\I$.
Also, $h(e)\in M$ if and only if both $f(e)\in M$ and $g(e)\in M$, as
desired.
\end{proof}

\begin{lem}\label{intersections-finitecase}
Let $M$ be a maximal ideal of $D$ and $\R=\fring{E}{D}$ a ring
of functions such that every $f\in\R$ takes values in only finitely
many residue classes mod $M$.

Then for every ideal $\mathcal{I}$ of $\R$,
$\Z_M(\I)$ is closed under finite intersections.
\end{lem}

\begin{proof}
Again, given $f,g\in\I$, we show that there exists $h\in\I$
with 
\[ h^{-1}(M) = f^{-1}(M)\cap g^{-1}(M) . \] 

Let $A,B\subseteq D/M$ be finite sets of residue classes of $D$ mod $M$
such that $f(E)$ is contained in the union of $A$
and $g(E)$ in the union of $B$. 

We can interpolate any function from $(D/M)^2$ to $(D/M)$ at any
finite set of arguments by a polynomial in $(D/M)[x,y]$. Pick
$\bar s\in (D/M)[x,y]$ with $\bar s(0,0)=0$ and $\bar s(a,b)=1$ 
for all $(a,b)\in (A\times  B)\setminus \{(0,0)\}$.
Let $s\in D[x,y]$ be a polynomial with zero constant coefficient that 
reduces to $\bar s$ when the coefficients are taken mod $M$.

Now, given $f$ and $g$ in $\I$, we set $h=s(f,g)$. 
By the fact that $\R$ contains all constant functions, $h$ is in $\I$.
Also, $h(e)\in M$ if and only if both $f(e)\in M$ and $g(e)\in M$, as
desired.
\end{proof}

\begin{Def}
Let $\R=\fring{E}{D}$ be a ring of functions and $M$ an ideal of
$D$. We call $f\in\R$ an $M$-unit-valued function if $f(e)+M$ is
a unit in $D/M$ for every $e\in E$. 
\end{Def}

\begin{thm}\label{maxideal} 
Let $M$ be a maximal ideal of $D$ and $\I$ an ideal of 
$\R=\fring{E}{D}$. Assume that either $D/M$ is not
algebraically closed or that each function in $\R$ takes values 
in only finitely many residue classes mod $M$.

\begin{enumerate}
\item\label{equivalence}
$\I$ is contained in an ideal of the form $\MF$ for 
some filter $\F$ on $E$ if and only if $\I$ 
contains no $M$-unit-valued function.

\item\label{conditionalmax}
Every ideal $\Q$ of $\R$ that is maximal with respect to
not containing any $M$-unit-valued function is of the form
$\MU$ for some ultrafilter $\U$ on $E$.

\item\label{max}
In particular, every maximal ideal of $\R$ that does not 
contain any $M$-unit-valued function is of the form
$\MU$ for some ultrafilter $\U$ on $E$.
\end{enumerate}
\end{thm}

\begin{proof}
Ad (\ref{equivalence}).
If $\I$ is contained in an ideal of the form $\MF$, $\I$ cannot
contain any $M$-unit-valued function, because $\F$ doesn't contain
the empty set. 

Conversely, suppose that $\I$ does not contain any $M$-unit-valued function.
Then $\emptyset\notin\Z_M(\I)$.
By Lemmata~\ref{intersections-notalgcl} and \ref{intersections-finitecase},
$\Z_M(\I)$ is closed under finite intersections.
$\Z_M(\I)$, therefore, satisfies the finite intersection property.
By Remark~\ref{finite-intersection-property}, $\Z_M(\I)$ is
contained in a filter $\F$ on $E$. 
For this filter, $\I\subseteq \MF$, by Remark~\ref{ZMproperties}.

Ad (\ref{conditionalmax}). 
Suppose $\Q$ is maximal with respect to not 
containing any $M$-unit-valued function.
By (\ref{equivalence}), $\Q\subseteq \MF$ for some filter $\F$. 
Refine $\F$ to an ultrafilter $\U$.
Then, by Remark~\ref{obviousinclusions}, 
$\Q\subseteq M_\mathcal{F} \subseteq M_\mathcal{U}$, and
$\MU$ doesn't contain any $M$-unit-valued function.
Since $\Q$ is maximal with this property, $\Q=\MU$.

(\ref{max}) is a special case of (\ref{conditionalmax}). 
\end{proof}

\section{A dichotomy of maximal ideals}

In what follows, $D$ is always a domain with quotient field $K$,
$E\subseteq D$ and $\mathcal{R}=\fring{E}{D}$ a ring of functions 
from $E$ to $D$ as in Def.~\ref{funcring-def}. When the interpretation 
of $\mathcal{R}$ as a subring of $\prod_{e\in E}D$ is understood, 
then for $M\subseteq D$ we let 
\[\fring{E}{M}=\{f\in\mathcal{R}\mid f(E)\subseteq M\}.\]

\begin{prop}\label{dichotomy} 
Let $M$ be a maximal ideal of $D$ and $\mathcal{Q}$ a maximal ideal 
of $\mathcal{R}=\fring{E}{D}$. 
Then exactly one of the following two statements holds:
\begin{enumerate}
\item\label{d1}
$\mathcal{Q}$ contains 
$\fring{E}{M}=\{f\in\mathcal{R}\mid f(E)\subseteq M\}$
\item\label{d2}
$\mathcal{Q}$ contains an
element $f$ with $f(e)\equiv 1$ mod $M$ for all $e\in E$.
\end{enumerate}
\end{prop}

\begin{proof}
The two cases are mutually exclusive, because any ideal $\mathcal{Q}$
satisfying both statements must contain $1$.

Now suppose $\mathcal{Q}$ does not contain $\fring{E}{M}$. Let 
$g\in \fring{E}{M}\setminus \mathcal{Q}$. By the maximality of
$\mathcal{Q}$, $1=h(x)g(x)+f(x)$ for some $h\in\mathcal{R}$ and
$f\in \mathcal{Q}$.
We see that
$f(x)=1-h(x)g(x)\in \mathcal{Q}$ satisfies $f(e)\equiv 1$ mod $M$ 
for all $e\in E$.
\end{proof}

Recall that a function $f\in\R$ is called $M$-unit-valued if
$f(e)+M$ is a unit in $D/M$ for every $e\in E$.

\begin{lem}\label{unit-valued-equivalence}
Let $M$ be an ideal of $D$ and
$\mathcal{Q}$ an ideal of $\mathcal{R}=\fring{E}{D}$.
The following are equivalent:
\begin{description}
\item[(A)]\label{e1}
$\mathcal{Q}$ contains an
element $f$ with $f(e)\equiv 1$ mod $M$ for all $e\in E$.
\item[(B)]\label{e2}
$\mathcal{Q}$ contains an $M$-unit-valued function that 
takes values in only finitely many residue classes mod $M$.
\end{description}
\end{lem}

\begin{proof}
To see that the a priori weaker statement implies the stronger,
let $g\in\mathcal{Q}$ be an $M$-unit-valued function taking only
finitely many different values mod $M$. Let $d_1,\ldots, d_k\in D$ 
be representatives of the finitely many residue classes mod $M$ 
intersecting $g(E)$ non-trivially, and $u\in D$ an inverse mod $M$
of $(-1)^{k+1}d_1\cdot\ldots\cdot d_k$.

Then
\[h(x)=\prod_{i=1}^k(g(x)-d_i) - (-1)^k d_1\cdot\ldots\cdot d_k\]
is in $\mathcal{Q}$ and 
$h(e)\equiv (-1)^{k+1} d_1\cdot\ldots\cdot d_k$ mod $M$
for all $e\in E$. Therefore $f(x)=uh(x)\in \mathcal{Q}$ satisfies 
$f(e)\equiv 1$ mod $M$ for all $e\in E$.
\end{proof}

\begin{prop}\label{finitecase} 
Let $M$ be a maximal ideal of $D$ and 
$\mathcal{Q}$ a maximal ideal of $\mathcal{R}=\fring{E}{D}$.
If each $f\in \mathcal{R}$ takes values in only finitely many 
residue classes mod $M$ 
\textrm{(}in particular, if $D/M$ happens to be finite\textrm{)}
then exactly one of the following statements holds:
\begin{enumerate}
\item\label{f1}
$\mathcal{Q}$ contains 
$\fring{E}{M}=\{f\in\mathcal{R}\mid f(E)\subseteq M\}$
\item\label{f2}
$\mathcal{Q}$ contains an $M$-unit-valued function.
\end{enumerate}
\end{prop}

\begin{proof} This follows directly from 
Proposition~\ref{dichotomy} and Lemma~\ref{unit-valued-equivalence}.
\end{proof}

The Propositions in this section partition the maximal ideals of $\R$ 
lying over a maximal ideal $M$ of $D$ into two types: those containing 
$\fring{E}{M}$ (the kernel of the restriction to $\R$ of the canonical 
projection
$\pi\colon \prod_{e\in E}D\longrightarrow \prod_{e\in E}(D/M)$),
and the others.

In some cases, it is known that all maximal ideals of $\R$ lying
over $M$ contain $\fring{E}{M}$, notably if $\mathcal{R}=\Int(D)$ 
and $M$ is finitely generated and of finite index in $D$ 
\cite[Ch.~V]{CahCha97IVP}, \cite[Lemma~4.4]{Fri13IVA}. We will
find a sufficient condition for all maximal ideals of $\R$ lying
over $M$ to contain $\fring{E}{M}$ in Theorem~\ref{REPcontainedinQ}.

We must not discount the possibility of a maximal ideal $\mathcal{Q}$ 
lying over $M$ containing an $M$-unit-valued function, however. 
If $D$ is an infinite domain, $D[x]$ is embedded in $D^D$ by mapping 
every polynomial to the corresponding polynomial function. 
When $D/M$ is not algebraically closed, then there are certainly 
maximal ideals of $D[x]$ lying over $M$ that contain polynomials 
without a zero mod $M$.

\section{Prime ideals containing $\R(E,M)$}

We are now in a position to characterize the prime ideals of
$\R$ containing $\R(E,D)$ as being precisely the ideals of the
form $\MU$ for ultrafilters $\U$ on $E$, under the following
hypothesis: every $f\in\R$ takes values in only finitely many
residue classes of $M$. 

This hypothesis may seem only marginally weaker than the assumption
that $D/M$ is finite. Note however, that it is sometimes satisfied 
for infinite $D/M$ under perfectly natural circumstances, for instance,
when $E$ intersects only finitely many residue classes of $M^n$ for
each $n\in\natn$ ($E$ precompact), and $\R$ consists of functions that 
are uniformly $M$-adically continuous.


As in the case of integer-valued polynomials, we can show that
every prime ideal of $\fring{E}{D}$ containing $\fring{E}{M}$ is
maximal under certain conditions, notably if $D/M$ is finite.
The proof for $\Int(D)$, when $D/M$ is finite 
\cite[Lemma~V.1.9.]{CahCha97IVP}, carries over practically without change.
Note that Definition~\ref{funcring-def} ensures that every ring of
functions $\R$ contains all constant functions -- an essential requirement
of the following proof.

\begin{lem}\label{primesmaximal}
Let $M$ be a maximal ideal of $D$ such that every function in 
$\R=\fring{E}{D}$ takes values in only finitely many residue 
classes mod $M$, and $\Q$ a prime ideal of $\fring{E}{D}$ 
containing $\fring{E}{M}$. Then $\Q$ is maximal and $\R/\Q$
is isomorphic to $D/M$.
\end{lem}

\begin{proof}
Let $\Q$ be a prime ideal of $\fring{E}{D}$ containing
$\fring{E}{M}$, and $A$ a system of representatives of $D$ mod $M$.
It suffices to show that $A$ (viewed as a set of constant functions)
is also a system of representatives of $\R$ mod $\Q$. 
Let $f\in \fring{E}{D}$ and $a_1,\ldots,a_r\in A$ the representatives 
of those residue classes of $M$ that intersect $f(E)$ non-trivially.
Then $\prod_{i=1}^r(f-a_i)$ is in $\fring{E}{M}\subseteq \Q$ and, 
$\Q$ being prime, one of the factors $(f-a_i)$ must be in $\Q$. 
This shows that $f$ is congruent mod $\Q$ to one of the constant functions 
$a_1,\ldots,a_r$, and, in particular, to an element of $A$.
Therefore, $A$ is a system of representatives of $\fring{E}{D}$ mod $\Q$.
\end{proof}

\begin{lem} \label{finitelem}
Let $\mathcal{R}=\fring{E}{D}$ a ring of functions and $M$ a 
maximal ideal of $D$ such that every $f\in \mathcal{R}$ takes values
in only finitely many residue classes of $M$. Let $\I$ be an ideal of $\R$.

Then $\I$ is contained in an ideal of the form $\MF$ 
for a filter $\F$ on $E$ 
if and only if $\fring{E}{M}\subseteq \I$.
\end{lem}

\begin{proof}
$\fring{E}{M}\subseteq \I$ is equivalent to $\I$ not containing
an $M$-unit-valued function, by Proposition~\ref{finitecase}.
The statement therefore follows from 
part (1) of Theorem~\ref{maxideal}.
\end{proof}

\begin{thm} \label{finitethm}
Let $\mathcal{R}=\fring{E}{D}$ a ring of functions,
and $M$ a maximal ideal of $D$. 
If every $f\in \mathcal{R}$ takes values in only finitely many
residue classes of $M$ \textrm{(}and, in particular, if $D/M$ 
is finite\textrm{)}, then
the prime ideals of $\R$ containing $\fring{E}{M}$ are exactly
the ideals of the form $\MU$ with $\U$ an ultrafilter on $E$.
Each of them is maximal and its residue field isomorphic to $D/M$.
\end{thm}

\begin{proof}
Let $\Q$ be a prime ideal of $\R$ containing $\fring{E}{M}$.
By Lemma~\ref{primesmaximal}, $\Q$ is maximal and $\R/\Q$ is 
isomorphic to $D/M$. By Lemma~\ref{finitelem}, $\Q\subseteq \MF$
for some filter $\F$ on $E$. $\F$ can be refined to an ultrafilter
$\U$ on $E$, and then $\Q\subseteq \MF\subseteq \MU\ne\R$,
by Remark~\ref{obviousinclusions}. 
Since $\Q$ is maximal, $\Q=\MU$ follows.

Conversely, every ideal of the form $\MU$ for an ultrafilter
$\U$ on $E$ is prime, by Lemma~\ref{ufprimes-prime}, and 
contains $\fring{E}{M}$, by Remark~\ref{obviousinclusions}.
\end{proof}


Note, in particular, that Theorems~\ref{maxideal} and
\ref{finitethm} apply to $\R=\Int(E,D)$.
In this way, we see, when $M$ is a maximal ideal of finite index in $D$,
that prime ideals of $\Int(E,D)$ containing $\Int(D,M)$ are
inverse images of prime ideals of $D^E$, and ultimately come 
from ultrapowers of $(D/M)$, as in the discussion after 
Lemma~\ref{ufprimes-prime}.

\section{Divisible rings of functions}

Let $\R\subseteq D^E$ be a ring of functions and $M$ a maximal ideal
of $D$. We have seen that we can describe those maximal ideals of $\R$
lying over $M$ that contain $\fring{E}{M}$. 
We would like to know under what conditions this holds for 
every maximal ideal of $\R$ lying over $M$.

In the case where $M$ is a maximal ideal of finite index in a 
one-dimensional Noetherian domain $D$, Chabert showed that every 
maximal ideal of $\Int(D)$ lying over $M$ contains $\Int(D,M)$,
cf.~\cite[Prop.~V.1.11]{CahCha97IVP} and \cite[Lemma~3.3]{Fri13IVA}.
Once we know this, Theorem~\ref{finitethm} is applicable. It can 
be used to give an alternative proof of the fact that every prime
ideal of $\Int(D)$ lying over $M$ is maximal and of the form 
$M_\alpha=\{f\in \Int(D)\mid f(\alpha)\in \hat M\}$ for an element 
$\alpha$ in the $M$-adic completion of $D$.

We will now generalize Chabert's argument from integer-valued polynomials
to a class of rings of functions which we call divisible. Note that we
do not have to restrict ourselves to Noetherian domains; we only require
the individual maximal ideal for which we study the primes of $\R$
lying over it to be finitely generated. It is true that our questions 
only localize well when the domain is Noetherian, but we will pursue
a different course, not relying on localization.

\begin{Def}\label{divisibledef}
Let $R$ be a commutative ring and $E\subseteq R$.
We call a ring of functions $\R\subseteq R^E$
\textbf{divisible} if it has the following property:
If $f\in\R$ is such that $f(E)\subseteq cR$ for some non-zero $c\in R$,
then every function $g\in R^E$ satisfying $cg(x)=f(x)$ is 
also in $\R$.

We call $\R$ \textbf{weakly divisible} if for every 
$f\in\R$ and every non-zero $c\in R$ such that $f(E)\subseteq cR$,
there exists a function $g\in\R$ with $cg(x)=f(x)$. 
\end{Def}

If $R$ is a domain, we note that $g(x)$ in the above definition 
is unique and that, therefore, for domains,
weakly divisible is equivalent to divisible.

\begin{exa}
\begin{enumerate}
\par
\item
$\Int(E,D)$ is divisible. - This is our motivation.
\item
If $D$ is a valuation domain with maximal ideal $M$ then the ring of 
uniformly $M$-adically continuous functions from $E$ to $D$ ($E\subseteq D$
equipped with subspace topology of $M$-adic topology) is a divisible ring 
of functions.
\end{enumerate}
\end{exa}

We now consider minimal prime ideals of  non-zero principal ideals,
that is, $P$ containing some $p\ne 0$ such that there is no prime 
ideal strictly contained in $P$ and containing $p$.
If $D$ is Noetherian, this condition reduces to ``$\height(P)=1$''. In
non-Noetherian domains, we find examples with $\height(P)>1$, for
instance, the maximal ideal of a finite-dimensional valuation domain.

\begin{lem}\label{pseudoprincipal}
Let $R$ be a domain, $P$ a finitely generated prime ideal that is
a minimal prime of a non-zero principal ideal $(p)\subseteq P$. 
Then there exist $m\in \natn$ and $s\in R\takeaway P$ such that 
$sP^m\subseteq pR$.
\end{lem}

\begin{proof}
In the localization $R_P$, $P_P$ is the radical of $pR_P$.
Therefore, since $P$ (and hence $P_P$) is finitely generated, 
there exists $m\in\natn$ with ${P_P}^m\subseteq pR_P$ and in particular
$P^m\subseteq pR_P$.
The ideal $P^m$ is also finitely generated, by $p_1,\ldots,p_k$, say.
Let $a_i\in R_P$ with $p_i=pa_i$. By considering the fractions
$a_i=r_i/s_i$ (with $r_i\in R$ and $s_i\in R\takeaway P$), and 
setting $s=s_1\cdot\ldots\cdot s_k$, we see that $sP^m\subseteq pR$ 
as desired.
\end{proof}

\begin{thm}\label{REPcontainedinQ}
Let $D$ be a domain and $P$ a finitely generated prime ideal
that is a minimal prime of a non-zero principal ideal.
Let $\R\subseteq D^E$ be a divisible ring of functions from $E$ to $D$.
Then every prime ideal $\Q$ of $\R$ with $\Q\cap D=P$ contains 
$\fring{E}{P}$.
\end{thm}

\begin{proof}
Let $f\in \fring{E}{P}$. Let $p\in P$ non-zero and such that 
there is no prime ideal $P_1$ with $(p)\subseteq P_1\subsetneq P$.
By Lemma~\ref{pseudoprincipal},
there are $m\in\natn$ and $s\in D\takeaway P$ such
that $sP^m\subseteq pD$.
Then $sf^m\in \fring{E}{pD}$. Since $\R$ is divisible,
$sf^m=pg$ for some $g\in \fring{E}{D}$. 
Therefore, $sf^m\in p\,\fring{E}{D}\subseteq \Q$.
As $\Q$ is prime and $s\notin \Q$, we conclude that $f\in \Q$.
\end{proof}

\begin{cor}\label{finalcor}
Let $D$ be a domain, $M$ a finitely generated maximal
ideal of height $1$, and $E$ a subset of $D$. Let $\R\subseteq D^E$ be
a divisible ring of functions from $E$ to $D$, such that each $f\in\R$
takes its values in only finitely many residue classes of $M$ in $D$.

Then the prime ideals of $\R$ lying over $M$ are precisely the ideals
of the form $M_\U$ for an ultrafilter $\U$ on $E$. Each $M_\U$ is
a maximal ideal and its residue field isomorphic to $D/M$.
\end{cor}

\begin{proof}
This follows from Theorem~\ref{REPcontainedinQ} via
Theorem~\ref{finitethm}.
\end{proof}

To summarize, we can, using ultrafilters, describe certain prime 
ideals of a ring of functions $\R=\fring{E}{D}$ lying over a maximal 
ideal $M$ pretty well: namely, those prime ideals that do not contain 
$M$-unit-valued functions (Theorem~\ref{maxideal}), or that contain
$\fring{E}{M}$ (Theorem~\ref{finitethm}).

We have, so far, little information about when all prime ideals of
$\R$ lying over $M$ are of this form, apart from the sufficient
condition in Theorem~\ref{REPcontainedinQ}. 

If we restrict our attention to rings of functions $\R$ with 
$D[x]\subseteq\fring{E}{D}\subseteq D^E$,
it would be interesting to find a precise criterion, perhaps involving
topological density, for this property.

Note that in the ``nicest'' case, that of $\Int(D)$, 
where $D$ is a Dedekind ring with finite residue fields, not only is
$\Int(D,M)$ contained in every prime ideal of $\Int(D)$ lying over
a maximal ideal $M$ of $D$, but also $\Int(D)$ is dense in $D^D$
with product topology of discrete topology on $D$ 
\cite{Fri99intp,CahChaFri00intp}. 


\bibliography{ufprimes}

\begin{thebibliography}{1}

\bibitem{CahCha97IVP}
{\sc P.-J. Cahen and J.-L. Chabert}, {\em Integer-valued polynomials}, vol.~48
  of Mathematical Surveys and Monographs, American Mathematical Society,
  Providence, RI, 1997.

\bibitem{CahChaFri00intp}
{\sc P.-J. Cahen, J.-L. Chabert, and S.~Frisch}, {\em Interpolation domains},
  J. Algebra 225 (2000), 794--803.

\bibitem{Fri99intp}
{\sc S.~Frisch}, {\em Interpolation by integer-valued polynomials}, J. Algebra
  211 (1999), 562--577.

\bibitem{Fri13IVA}
{\sc S.~Frisch}, {\em Integer-valued polynomials on algebras}, J. Algebra 373
  (2013), 414--425, Corrigendum: 412 (2014) p282.

\bibitem{GilHei94Imb}
{\sc R.~Gilmer and W.~Heinzer}, {\em Imbeddability of a commutative ring in a
  finite-dimensional ring}, Manuscripta Math. 84 (1994), 401--414.

\bibitem{GoJu95IncPh}
{\sc M.~Goldstern and H.~Judah}, {\em The incompleteness phenomenon}, A K
  Peters, Ltd., Wellesley, MA, 1995.
\newblock A new course in mathematical logic, With a foreword by Saharon
  Shelah.

\bibitem{Hin05mathlog}
{\sc P.~G. Hinman}, {\em Fundamentals of mathematical logic}, A K Peters, Ltd.,
  Wellesley, MA, 2005.

\bibitem{LeLoSh91PSIPZ}
{\sc R.~Levy, P.~Loustaunau, and J.~Shapiro}, {\em The prime spectrum of an
  infinite product of copies of {${\bf Z}$}}, Fund. Math. 138 (1991), 155--164.

\end{thebibliography}
\bibliographystyle{siamese}

\end{document}